\documentclass{amsart}

\usepackage{amscd}
\usepackage{amssymb}

\usepackage[T1]{fontenc}
\usepackage[latin1]{inputenc}
\usepackage{hyperref}
\usepackage[arrow,curve,matrix]{xy}
\usepackage{times}
\usepackage{graphicx}
\usepackage{color}
\usepackage{flafter}
\usepackage{placeins}
\usepackage{enumitem}

\usepackage[dpi=1270,PostScript=dvips,heads=LaTeX]{diagrams}
\diagramstyle{small,labelstyle=\scriptstyle}

\def\OO{{\mathcal O}}

\def\D{\mathbf{D}}

\def\Pic{{\rm Pic}}
\def\Aut{{\rm Aut}}

\def\id{{\rm id}}
\def\Aff{{\rm Aff}}

\newcommand{\FM}[1]{\Phi_{#1}}

\newcommand{\alb}{\textnormal{alb}}

\newcommand{\Alb}{\textnormal{Alb}}

\newcommand{\FF}{\mathcal{F}}
\newcommand{\EE}{\mathcal{E}}

\theoremstyle{plain}

\newtheorem{theorem}{Theorem}[section]
\newtheorem{theoremalpha}{Theorem}
\newtheorem{corollaryalpha}[theoremalpha]{Corollary}

\newtheorem{proposition/example}[theorem]{Proposition/Example}

\newtheorem{corollary}[theorem]{Corollary}
\newtheorem{lemma}[theorem]{Lemma}

\theoremstyle{definition}

\newtheorem{remark}[theorem]{Remark}

\newtheorem{conjecture/question}[theorem]{Conjecture/Question}

\newtheorem{remark/definition}[theorem]{Remark/Definition}
\newtheorem{definition/notation}[theorem]{Definition/Notation}

 \theoremstyle{remark}

\numberwithin{equation}{section}

\newcommand{\Ker}{\operatorname{Ker}}
\renewcommand{\Im}{\operatorname{Im}}

\newcommand{\menge}[2]{\bigl\{ \thinspace #1 \thinspace\thinspace \big\vert%
\thinspace\thinspace #2 \thinspace \bigr\}}

\keywords{Derived categories, Picard variety, Hodge numbers}
\subjclass[2000]{14F05, 14K30}

\begin{document}

\title[Derived invariance of the number of holomorphic $1$-forms]{Derived invariance of the 
number of holomorphic $1$-forms and vector fields}

\author{Mihnea Popa}
\address{Department of Mathematics, University of Illinois at Chicago,
851 S. Morgan Street, Chicago, IL 60607, USA } \email{{\tt
mpopa@math.uic.edu}}
\thanks{First author partially supported by NSF grant DMS-0758253 and a Sloan Fellowship}

\author{Christian Schnell}
\address{Department of Mathematics, University of Illinois at Chicago,
851 S. Morgan Street, Chicago, IL 60607, USA } \email{{\tt
cschnell@math.uic.edu}}



\setlength{\parskip}{.09 in}

\maketitle

\section{introduction}

Given a smooth projective variety $X$, we denote by $\D(X)$ the bounded derived 
category of coherent sheaves $\D^{{\rm b}} ({\rm Coh}(X))$.  All varieties we consider below are over the complex 
numbers. 
A result of Rouquier, \cite{rouquier} Th\'eor\'eme 4.18, asserts that if  $X$ and $Y$ are smooth projective varieties with 
$\D(X) \simeq \D(Y)$ (as linear triangulated categories), then there is an isomorphism of algebraic groups
$$\Aut^0 (X) \times \Pic^0 (X) \simeq  \Aut^0 (Y) \times \Pic^0 (Y).$$
We refine this by showing that each of the two factors is almost invariant under derived equivalence. 
According to Chevalley's theorem $\Aut^0 (X)$, the connected component of the identity in $\Aut (X)$, has a unique maximal connected affine subgroup $\Aff(\Aut^0(X))$, and the quotient $\Alb(\Aut^0 (X))$ by this subgroup is an abelian variety,  the Albanese variety of 
$\Aut^0(X)$. The affine parts $\Aff(\Aut^0(X))$ and $\Aff(\Aut^0(Y))$, being also the affine parts of the two sides in the isomorphism above, are isomorphic. The main result of the paper is

\begin{theoremalpha}\label{isogeny}
Let $X$ and $Y$ be smooth projective varieties such that $\D(X) \simeq \D(Y)$. Then 

\noindent
(1) ${\rm Pic}^0 (X)$ and ${\rm Pic}^0 (Y)$ are isogenous; equivalently, $\Alb(\Aut^0 (X))$ and $\Alb(\Aut^0 (Y))$ are isogenous.

\noindent
(2) ${\rm Pic}^0 (X) \simeq {\rm Pic}^0 (Y)$ unless $X$ and $Y$ are \'etale locally trivial fibrations over isogenous 
positive dimensional abelian varieties (hence $\chi (\OO_X) = \chi (\OO_Y)= 0$). 
\end{theoremalpha}

The key content is part (1), while (2) simply says that $\Aut^0 (X)$ and $\Aut^0 (Y)$ are affine unless the geometric condition stated there holds (hence the presence of abelian varieties is essentially the only reason for the failure of the derived invariance of the Picard variety).

\begin{corollaryalpha}\label{irregularity}
If $\D(X) \simeq \D(Y)$, then
$$h^0 (X, \Omega_X^1) = h^0 (Y, \Omega_Y^1) \quad \text{and}\quad  h^0 (X, T_X) = h^0 (Y, T_Y).$$
\end{corollaryalpha}

The Hodge number $h^{1,0}(X)= h^0 (X, \Omega_X^1)$ is also called the
\emph{irregularity} $q(X)$, the dimension of the Picard and Albanese varieties
of $X$. The invariance of the sum $h^0 (X, \Omega_X^1) +  h^0 (X, T_X)$ was already
known, and is a special case of the derived invariance of the Hochschild \emph{co}homology
of $X$ (\cite{orlov}, \cite{caldararu}; cf. also \cite{huybrechts} \S6.1). Alternatively, it follows from Rouquier's result above.
Corollary \ref{irregularity}, together with the derived invariance of Hochschild
homology (cf. \emph{loc. cit.}), implies the invariance of all Hodge numbers for all derived equivalent
threefolds. This was expected to hold as suggested by work of Kontsevich \cite{kontsevich} (cf. also \cite{bk}).

\begin{corollaryalpha}\label{threefolds}
Let $X$ and $Y$ be smooth projective threefolds with $\D(X) \simeq \D(Y)$.  Then
$$h^{p,q}(X) = h^{p,q}(Y)$$
for all $p$ and $q$.
\end{corollaryalpha} 
\begin{proof}
The fact that the Hochschild homology of $X$ and $Y$ is the same gives 
\begin{equation}\label{hochschild}
\underset{p-q= i}{\sum} h^{p,q}(X) = \underset{p-q= i}{\sum} h^{p,q}(Y).
\end{equation}
for all $i$. A straightforward calculation shows that this implies the invariance of all Hodge numbers 
except for $h^{1,0}$ and $h^{2,1}$, about which we only get that 
$h^{1,0}+ h^{2,1}$ is invariant. We then apply Corollary \ref{irregularity}.
\end{proof}

Corollary \ref{threefolds} is already known (in arbitrary dimension) for varieties of general
type: for these derived equivalence implies $K$-equivalence by a result of Kawamata
\cite{kawamata}, while $K$-equivalent varieties have the same Hodge numbers according to 
Batyrev \cite{batyrev} and Kontsevich, Denef-Loeser \cite{dl}. It is also well known for Calabi-Yau threefolds; 
more generally it follows easily for threefolds with numerically trivial canonical 
bundle (condition which is preserved by derived equivalence, see \cite{kawamata} Theorem 1.4). 
Indeed, since for threefolds Hirzebruch-Riemann-Roch gives $\chi (\omega_X) = \frac{1}{24} c_1 (X) c_2(X)$,  in this case $\chi (\omega_X)= 0$, hence $h^{1,0}(X)$ can be expressed in terms of 
Hodge numbers that are known to be derived invariant as above. Finally, in general the invariance of $h^{1,0}$ would follow automatically if $X$ and $Y$ were birational, but derived equivalence does not necessarily imply birationality. 

The proof of Theorem \ref{isogeny} in \S3 uses a number of standard facts in the study of derived
equivalences: invariance results and techniques due to Orlov and Rouquier, Mukai's description of semi-homogeneous vector bundles, and Orlov's fundamental characterization of derived equivalences. 
The main new ingredients are results of Nishi-Matsumura and Brion on actions of non-affine algebraic groups (see \S2). 
Further numerical applications of Corollary \ref{irregularity} to fourfolds or abelian varieties are provided in Remark \ref{num_ap}.

Finally, the case of abelian varieties shows the existence of Fourier-Mukai partners with non-isomorphic Picard varieties. We expect however the following stronger form of 
Theorem \ref{isogeny}(1).

\noindent
{\bf  Conjecture.} 
\emph{If $\D(X) \simeq \D(Y)$, then $\D({\rm Pic}^0 (X)) \simeq \D({\rm Pic}^0 (Y))$.}

\noindent
Derived equivalent curves must be isomorphic (see e.g. \cite{huybrechts}, Corollary 5.46), while in the case of surfaces the 
conjecture is checked in the upcoming thesis of Pham \cite{pham} using the present methods and 
the classification of Fourier-Mukai equivalences in \cite{bm} and \cite{kawamata}.

\noindent
{\bf Acknowledgements.} We thank A. C\u ald\u araru, L. Ein, D. Huybrechts and M. Musta\c t\u a for useful comments, and a 
referee for suggesting improvements to the exposition.

\section{Actions of non-affine algebraic groups}

Most of the results in this section can be found in Brion \cite{brion1}, \cite{brion2}, or are at least implicit there.
Let $G$ be a connected algebraic group. According to Chevalley's theorem (see e.g. \cite{brion1} p.1), 
$G$ has a unique maximal connected affine subgroup $\Aff(G)$, and the quotient $G / \Aff(G)$ is an
abelian variety. We denote this abelian variety by $\Alb(G)$, since the map $G \to
\Alb(G)$ is the Albanese map of $G$, i.e.\@ the universal morphism to an abelian variety (see \cite{serre2}). Thus $G \to \Alb(G)$ is a homogeneous fiber bundle with fiber $\Aff(G)$. 

\begin{lemma}[\cite{brion2}, Lemma 2.2]\label{loc_triv}
The map $G \to \Alb(G)$ is locally trivial in the Zariski topology.
\end{lemma}

Now let $X$ be a smooth projective variety. We abbreviate $G_X := \Aut^0(X)$, and let $a(X)$
be the dimension of the abelian variety $\Alb(G_X)$. The group $G_X$ naturally acts
on the Albanese variety $\Alb(X)$ as well (see \cite{brion1} \S3).

\begin{lemma}\label{action}
The action of $G_X$ on $\Alb(X)$ induces a map of abelian varieties 
$$\Alb(G_X) \to \Alb(X),$$
whose image is contained in the Albanese image ${\rm alb}_X (X)$. 
More precisely, the composition $G_X \to \Alb(X)$ is given by the formula $g
\mapsto \alb_X (g x_0 - x_0)$, where $x_0 \in X$ is an arbitrary point.
\end{lemma}

\begin{proof}
From $G_X \times X \to X$, we obtain a map of abelian varieties 
\[
	\Alb(G_X) \times \Alb(X) \simeq \Alb(G_X \times X) \to \Alb(X).
\]
It is clearly the identity on $\Alb(X)$, and therefore given by a map of abelian
varieties $\Alb(G_X) \to \Alb(X)$. To see what it is, fix a base-point $x_0 \in X$, and
write the Albanese map of $X$ in the form $X \to \Alb(X)$, $x \mapsto \alb_X(x - x_0)$.
Let $g \in G_X$ be an automorphism of $X$. By the universal property of $\Alb(X)$, it
induces an automorphism $\tilde{g} \in \Aut^0 \bigl( \Alb(X) \bigr)$, making the
diagram
\begin{diagram}[width=3em]
X &\rTo^g& X \\
\dTo && \dTo \\
\Alb(X) &\rTo^{\tilde{g}}& \Alb(X)
\end{diagram}
commute; in other words, $\tilde{g} \bigl( \alb_X(x - x_0) \bigr) = \alb_X(gx - x_0)$.
Any such automorphism is translation by an element of $\Alb(X)$, and the formula
shows that this element has to be $\alb_X(g x_0 - x_0)$. It follows that the map $G_X \to
\Alb(X)$ is given by $g \mapsto \alb_X(g x_0 - x_0)$. By Chevalley's theorem, it
factors through $\Alb(G_X)$.
\end{proof}

\noindent
A crucial fact is the following theorem of  Nishi and Matsumura (cf. also \cite{brion1}).

\begin{theorem}[\cite{matsumura}, Theorem 2]\label{nishi}
The map $\Alb(G_X) \to \Alb(X)$ has finite kernel. More generally, any connected algebraic group 
$G$ of automorphisms of $X$ acts on  $\Alb(X)$ by translations, and the kernel of the induced 
homomorphism $G \rightarrow \Alb(X)$ is affine.
\end{theorem}

Consequently, the image of $\Alb(G_X)$ is an abelian subvariety of $\Alb(X)$ of dimension $a(X)$.
This implies the inequality $a(X) \leq q(X)$. Brion observed that $X$ can always be
fibered over an abelian variety which is a quotient of $\Alb(G_X)$ of the same
dimension $a(X)$; the following proof is taken from \cite{brion1}, p.2 and \S3, and
is included for later use of its ingredients.

\begin{lemma}\label{fibration}
There is an affine subgroup $\Aff(G_X) \subseteq H \subseteq G_X$ with $H / \Aff(G_X)$
finite, such that $X$ admits a $G_X$-equivariant map $\psi \colon X \to G_X/H$.
Consequently, $X$ is isomorphic to the equivariant fiber bundle $G_X \times^H Z$
with fiber $Z = \psi^{-1}(0)$.
\end{lemma}

\begin{proof}
By the Poincar\'e complete reducibility theorem, the map $\Alb(G_X) \to \Alb(X)$
splits up to isogeny. This means that we can find a subgroup $H$ containing
$\Aff(G_X)$, such that there is a surjective map $\Alb(X) \to G_X / H$ with
$\Alb(G_X) \to G_X/H$ an isogeny. It follows that $H / \Aff(G_X)$ is finite, and hence that $H$
is an affine subgroup of $G_X$ whose identity component is $\Aff(G_X)$. Let $\psi \colon
X \to G_X/H$ be the resulting map; it is equivariant by construction. Since $G_X$ acts
transitively on $G_X/H$, we conclude that $\psi$ is an equivariant fiber bundle over
$G_X/H$ with fiber $Z = \psi^{-1}(0)$, and therefore isomorphic to
\[
	G_X \times^H Z = (G_X \times Z) / H,
\]
where $H$ acts on the product by $(g, z) \cdot h = (g\cdot h, h^{-1}\cdot z)$.
\end{proof}

Note that the group $H$ naturally acts on $Z$; the proof shows that we obtain $X$
from the principal $H$-bundle $G_X \to G_X/H$ by replacing the fiber $H$ by $Z$ (see
\cite{serre1}, \S3.2).  While $X \to G_X/H$ is not necessarily locally trivial,
it is so in the \'etale topology.

\begin{lemma}\label{isotrivial}
Both $G_X \to G_X/H$ and $X \to G_X/H$ are \'etale locally trivial.
\end{lemma}

\begin{proof}
Consider the pullback of $X$ along the \'etale map $\Alb(G_X) \to G_X/H$,
\begin{diagram}[width=3em]
X' &\rTo& X \\
\dTo && \dTo \\
\Alb(G_X) &\rTo& G_X/H.
\end{diagram}
One notes that $X' \to \Alb(G_X)$ is associated to the principal bundle
$G_X \to \Alb(G_X)$. The latter is locally trivial in the Zariski topology by Lemma \ref{loc_triv}.
\end{proof}

\begin{corollary}\label{chi}
If $a(X) > 0$ (i.e.\@ $G_X$ is not affine), then $\chi(\OO_X) = 0$.
\end{corollary}

\begin{proof}
Clearly $\chi(\OO_{X'}) = 0$ since $X'$ is locally isomorphic to the product of
$Z$ and $\Alb(G_X)$. But $\chi(\OO_{X'}) = \deg(X'/X) \cdot \chi(\OO_X)$.
\end{proof}

\section{Proof of the main result}

Let $\Phi \colon \D(X) \to \D(Y)$ be an exact equivalence between the derived
categories of two smooth projective varieties $X$ and $Y$. By Orlov's 
criterion, $F$ is uniquely up to isomorphism a Fourier-Mukai functor, i.e.\@ 
$\Phi \simeq \FM{\EE}$ with $\EE \in \D(X\times Y)$, where 
$\FM{\EE} (\cdot) = {p_Y}_* (p_X^* (\cdot) \otimes \EE)$. (Here and in what follows 
all functors are derived.) A result of Rouquier, \cite{rouquier}
Th\'eor\'eme~4.18 (see also \cite{huybrechts}, Proposition 9.45), says that $\Phi$ induces an
isomorphism of algebraic
groups\footnote{Note that in the quoted references the result is stated for the
semidirect product of $\Pic^0 (X)$ and $\Aut^0 (X)$.  One can however check 
that the action of $\Aut^0(X)$ on $\Pic^0(X)$ is trivial. Indeed, $\Aut^0(X)$ acts on $\Pic^0 (X)$ 
by elements in $\Aut^0(\Pic^0(X))$, which are translations. Since the origin is fixed, these
must be trivial. This shows in particular that $\Aut^0(X)$ and $\Pic^0(X)$ commute as subgroups of 
$\Aut (\D(X))$.}
\begin{equation}\label{rouquier}
	F \colon \Aut^0(X) \times \Pic^0(X) \simeq \Aut^0(Y) \times \Pic^0(Y)
\end{equation}
in the following manner: A pair of $\varphi \in \Aut(X)$ and $L \in \Pic(X)$ defines
an auto-equivalence of $\D(X)$ by the formula $\varphi_*\bigl( L \otimes (\cdot) \bigr)$; 
its kernel is $(\id, \varphi)_{\ast} L \in \D(X \times X)$. When $(\varphi, L) \in
\Aut^0(X) \times \Pic^0(X)$, Rouquier proves that the composition $\FM{\EE} \circ
\FM{(\id, \varphi)_{\ast} L} \circ \FM{\EE}^{-1}$ is again of the form $\FM{(\id,
\psi)_{\ast} M}$ for a unique pair $(\psi, M) \in \Aut^0(Y) \times \Pic^0(Y)$. We
then have $F(\varphi, L) = (\psi, M)$. The following interpretation in terms of the kernel $\EE$
was proved by Orlov (see \cite{orlov}, Corollary 5.1.10) for abelian varieties; the general 
case is similar, and we include it for the reader's convenience.

\begin{lemma} \label{lem:Rouquier}
One has $F(\varphi, L) = (\psi, M)$ if and only if 
\[
	p_1^{\ast} L \otimes (\varphi\times \id)^{\ast} \EE \simeq
	p_2^{\ast} M \otimes (\id \times \psi)_{\ast} \EE. 
\]
\end{lemma}

\begin{proof}
By construction, $F(\varphi, L) = (\psi, M)$ is equivalent to the relation
$$\FM{\EE} \circ \FM{(\id, \varphi)_{\ast} L} = \FM{(\id, \psi)_{\ast} M} \circ \FM{\EE}.$$
Since both sides are equivalences, their kernels have to be isomorphic. Mukai's formula for the 
kernel of the composition of two integral functors (see \cite{huybrechts}, Proposition 5.10) gives
\begin{equation}\label{eq:rel}
{p_{13}}_* \bigl(p_{12}^* (\id , \varphi)_* L \otimes p_{23}^* \EE\bigr) \simeq
{p_{13}}_* \bigl(p_{12}^* \EE \otimes p_{23}^* (\id , \psi)_* M \bigr).
\end{equation}
To compute the left-hand side of \eqref{eq:rel}, let $\lambda \colon X \times Y \to X
\times X \times Y$ be given by $\lambda(x,y) = (x, \varphi(x), y)$, making the
following diagram commutative:
\begin{diagram}[l>=2em]
X \times Y &\rTo^{\lambda}& X \times X \times Y &\rTo^{p_{13}}& X \times Y \\
\dTo^{p_1} && \dTo^{p_{12}} \\
X &\rTo^{(\id,\varphi)}& X \times X
\end{diagram}
By the base-change formula, $p_{12}^* (\id,\varphi)_* L \simeq \lambda_* p_1^* L$; using the
projection formula and the identities $p_{13} \circ \lambda = \id$ and $p_{23} \circ
\lambda = \varphi \times \id$, we then have
\[
	{p_{13}}_*  \bigl( p_{12}^* (\id, \varphi)_* L \otimes p_{23}^* \EE \bigr) \simeq
		p_1^* L \otimes \lambda^* p_{23}^* \EE	 \simeq 
		p_1^* L \otimes (\varphi \times \id)^* \EE.
\]
To compute the right-hand side of \eqref{eq:rel}, we similarly define $\mu \colon X
\times Y \to X \times Y \times Y$ by the formula $\mu(x,y) = (x, y, \psi(y))$, to fit
into the diagram
\begin{diagram}[l>=2em]
X \times Y &\rTo^{\mu}& X \times Y \times Y &\rTo^{p_{13}}& X \times Y \\
\dTo^{p_2} && \dTo^{p_{23}} \\
Y &\rTo^{(\id,\psi)}& Y \times Y.
\end{diagram}
Since $p_{13} \circ \mu = (\id \times \psi)$ and $p_{12} \circ \mu = \id$, the same
calculation as above shows that
\[
	{p_{13}}_*  \bigl( p_{12}^* \EE \otimes p_{23}^* (\id, \psi)_* M \bigr) \simeq
	(\id \times \psi)_* \bigl( \EE \otimes p_2^* M \bigr) \simeq
	(\id \times \psi)_* \EE \otimes p_2^* M,
\]
where the last step uses that the action of $\Aut^0(Y)$ on $\Pic^0(Y)$ is trivial, 
so $(\id \times \psi)^* p_2^* M \simeq p_2^*M$.
\end{proof}

We now give the proof of Theorem \ref{isogeny}.  It is 
in fact more convenient to start directly with the numerical Corollary \ref{irregularity}.  
Note that Rouquier's result (or the invariance of the first Hochschild cohomology) implies the derived
invariance of the quantity $h^0 (X, \Omega_X^1) + h^0 (X, T_X)$.
Hence it suffices to show that $q(X) = q(Y)$, where we set $q(X) = h^0 (X, \Omega_X^1)$, 
and similarly for $Y$. 

We continue to write $G_X = \Aut^0(X)$ and $G_Y = \Aut^0(Y)$. Let $\EE$ be the
kernel defining the equivalence, and let $F \colon G_X \times \Pic^0(X) \to G_Y
\times \Pic^0(Y)$ be the isomorphism of algebraic groups 
from Rouquier's theorem, as above. To prove the assertion, we consider the map 
$$\beta \colon \Pic^0(X) \to G_Y, \quad \beta(L) = p_1 \bigl( F(\id, L) \bigr),$$ 
and let $B = \Im \beta$.
Similarly, we define 
$$\alpha \colon \Pic^0(Y) \to G_X, \quad \alpha(M) = p_1
\bigl( F^{-1}(\id, M) \bigr),$$ 
and let $A = \Im \alpha$. One easily verifies that $F$ induces an isomorphism
\[
	F \colon A \times \Pic^0(X) \to B \times \Pic^0(Y).
\]
If both $A$ and $B$ are trivial, we immediately obtain $\Pic^0(X) \simeq \Pic^0(Y)$.
Excluding this case from now on, we let the abelian variety $A \times B$ act on
$X \times Y$ by automorphisms.  Take a point $(x,y)$ in the support of the kernel
$\EE$, and consider the orbit map
\[
	f \colon A \times B \to X \times Y, \quad 	
		(\varphi, \psi) \mapsto \bigl( \varphi(x), \psi(y) \bigr).
\]
By Lemma \ref{action} and the Nishi-Matsumura Theorem \ref{nishi}, the induced map $A
\times B \to \Alb(X) \times \Alb(Y)$ has finite kernel. Consequently, the dual map
$f^*: \Pic^0(X) \times \Pic^0(Y) \to \widehat{A} \times \widehat{B}$ is surjective.

Now let $\FF := f^{\ast} \EE \in \D(A \times B)$; it is nontrivial by our choice of
$(x,y)$. For $F(\varphi, L) = (\psi, M)$, the formula in Lemma~\ref{lem:Rouquier} can be
rewritten in the more symmetric form (again using the fact that $\psi^* M \simeq M$):
\begin{equation} \label{eq:formula-EE}
	(\varphi \times \psi)^* \EE \simeq (L^{-1} \boxtimes  M) \otimes \EE.
\end{equation}
For $(\varphi, \psi) \in A \times B$, let $t_{(\varphi, \psi)} \in \Aut^0(A
\times B)$ denote translation by $(\varphi, \psi)$. The identity in
\eqref{eq:formula-EE} implies that
$t_{(\varphi, \psi)}^{\ast} \FF \simeq f^*(L^{-1} \boxtimes  M) \otimes \FF$, whenever
$F(\varphi, L) = (\psi, M)$. We introduce the map
\[
	\pi = (\pi_1, \pi_2) \colon 
		A \times \Pic^0(X) \to (A \times B) \times (\widehat{A} \times \widehat{B}), \quad
		\pi(\varphi, L) = \bigl( \varphi, \psi, L^{-1} \vert_A,  M \vert_B \bigr),
\]
where we write $L^{-1}\vert_A$ for the pull-back from $\Alb(X)$ to A, and same for $M$.
We can then write the identity above as
\begin{equation} \label{eq:formula-FF}
	t_{\pi_1(\varphi, L)}^* \FF \simeq \pi_2(\varphi, L) \otimes \FF.
\end{equation}

Since $\pi_1 \colon A \times \Pic^0(X) \to A \times B$ is surjective, it follows that
each cohomology object $H^i(\FF)$ is a semi-homogeneous vector bundle on $A \times
B$, and that $\dim(\Im \pi) \ge \dim A + \dim B$. On the other hand Mukai \cite{mukai}, Proposition~5.1,
shows that the semi-homogeneity of $H^i (\FF)$ is equivalent to the fact that the closed 
subset 
$$\Phi (H^i (\FF)) := \{ (x, \alpha) \in (A \times B) \times (\widehat{A} \times \widehat{B}) ~|~  
t_x^* H^i (\FF) \simeq H^i (\FF) \otimes \alpha \}$$ has dimension precisely $\dim A + \dim
B$. This implies that $\dim(\Im \pi) = \dim A + \dim
B$ (and in fact that $\Im \pi = \Phi^0 (H^i (\FF))$, the neutral component, for any $i$, though we will not use this; 
note that $\Phi$ is denoted $\Phi^0$, and $\Phi^0$ is denoted $\Phi^{00}$ in \cite{mukai}). Furthermore, we have
\begin{align*}
	\Ker(\pi) &= \menge{(\id, L) \in A \times \Pic^0(X)}%
		{\text{$F(\id, L) = (\id, M)$ and $L \vert_A \simeq \OO_A$ and $M \vert_B \simeq \OO_B$}} \\
	&\subseteq \menge{L \in \Pic^0(X)}{L \vert_A \simeq \OO_A} 
		= \Ker \bigl( \Pic^0(X) \to \widehat{A} \bigr).
\end{align*}
Now the surjectivity of $f^*$ implies in particular that the restriction map $\Pic^0(X) \to \widehat{A}$ is surjective, so we get $\dim(\Ker \pi) \leq q(X) - \dim A$, and therefore
\[
	\dim A + \dim B = \dim A + q(X) - \dim(\Ker \pi) \geq 2 \dim A.
\]
Thus $\dim A \leq \dim B$; by symmetry, $\dim A = \dim B$, and finally, $q(X) =
q(Y)$.  This concludes the proof of the fact that $\Pic^0 (X)$ and $\Pic^0 (Y)$ have
the same dimension. 

We now use this to show that they are in fact isogenous. Let $d = \dim A = \dim B$.
The reasoning above proves that $\Im \pi$ is an abelian subvariety of $(A \times B)
\times (\widehat{A} \times \widehat{B})$, with $\dim(\Im \pi) = 2d$. For dimension reasons, we also have
\begin{equation} \label{eq:neutral}
	(\Ker \pi)^0 \simeq \bigl( \Ker( \Pic^0(X) \to \widehat{A}) \bigr)^0 
	\simeq \bigl( \Ker( \Pic^0(Y) \to \widehat{B}) \bigr)^0,
\end{equation}
where the superscripts indicate neutral components. We claim that the projection $p
\colon \Im \pi \to A \times \widehat{A}$ is an isogeny (likewise for $B \times \widehat{B}$).
Indeed, a point in $p^{-1}(\id, \OO_A)$ is of the form
$\bigl( \id, \psi, \OO_A, M \vert_B \bigr)$, where $F(\id, L) =
(\psi, M)$ and $L \vert_A \simeq \OO_A$. By \eqref{eq:neutral}, a fixed multiple of
$(\id, L)$ belongs to $\Ker \pi$, and so $\Ker p$ is a finite set. It follows that
$\Im \pi$ is isogenous to both $A \times \widehat{A}$ and $B \times \widehat{B}$; consequently,
$A$ and $B$ are themselves isogenous.

To conclude the proof of part (1), note that we have extensions
\[
	0 \to \Ker \beta \to \Pic^0(X) \to B \to 0 \qquad \text{and} \qquad
	0 \to \Ker \alpha \to \Pic^0(Y) \to A \to 0.
\]
By definition, $\Ker \beta$ consists of those $L \in \Pic^0(X)$ for which $F(\id, L)
= (\id, M)$; obviously, $F$ now induces an isomorphism $\Ker \beta \simeq \Ker
\alpha$, and therefore $\Pic^0(X)$ and $\Pic^0(Y)$ are isogenous. Now by Rouquier's isomorphism (\ref{rouquier})
and the uniqueness of ${\rm Aff}(G)$ in Chevalley's theorem we have ${\rm
Aff}(G_X)\simeq {\rm Aff}(G_Y)$ and $$\Alb(G_X) \times \Pic^0(X) \simeq  \Alb(G_Y)
\times \Pic^0(Y).$$
Therefore we also have equivalently that $\Alb(G_X)$ and $\Alb(G_Y)$  are isogenous.

It remains to check part (2). Clearly $a(X) = a(Y)$.  If
$a(X) = 0$, we obviously have $\Pic^0(X) \simeq \Pic^0(Y)$. On the other hand, if
$a(X) >0$, Lemmas \ref{fibration} and \ref{isotrivial} show that $X$ can be written
as an \'etale locally trivial fiber bundle over a quotient of $\Alb(G_X)$ by a finite
subgroup, so an abelian variety isogenous to $\Alb(G_X)$. The same holds for $Y$ by symmetry. Note
that in this case we have $\chi (\OO_X) = \chi (\OO_Y) = 0$ by Corollary \ref{chi}.

\begin{remark}
Results of Mukai \cite{mukai}, \S5 and \S6, imply that each
$H^i(\FF)$ on $A \times B$ in the proof above has a filtration with simple semi-homogeneous quotients, 
all of the same slope, associated to the subvariety $\Im \pi$. In line with Orlov's 
work on derived equivalences of abelian varieties \cite{orlov} \S5, one may guess 
that these simple bundles induce derived equivalences between
$A$ and $B$, and that $\Im \pi$ induces an isomorphism between $A \times
\widehat{A}$ and $B \times \widehat{B}$, but we have not been able to prove this.
\end{remark}

\begin{remark}[{\bf Further numerical applications}]\label{num_ap}
In the case of fourfolds, in addition to the Hodge numbers that are equal due to the general invariance of Hochschild homology (namely $h^{3,0}$ and $h^{4,0}$), Corollary \ref{irregularity} implies:

\begin{corollary}\label{fourfolds}
Let $X$ and $Y$ be smooth projective fourfolds with $\D(X) \simeq \D(Y)$.  Then
$h^{2,1} (X) = h^{2,1} (Y)$. If in addition $\Aut^0(X)$ is not affine, then $h^{2,0} (X) = h^{2,0} (Y)$ 
and $h^{3,1}(X) = h^{3,1} (Y)$.
\end{corollary}
\begin{proof}
The analogue of (\ref{hochschild}) for fourfolds implies that $h^{2,1}$ is invariant if and only if $h^{1,0}$ 
is invariant, and $h^{2,0}$ is invariant if and only if $h^{3,1}$ is invariant. On the other hand, if $\Aut^0(X)$
is not affine, then $\chi(\OO_X) = 0$ (cf. Lemma \ref{chi}), which implies that $h^{2,0}$ is invariant if and 
only if $h^{1,0}$ is invariant. We apply Corollary \ref{irregularity}.
\end{proof}

It is also worth noting that Corollary \ref{irregularity} can help in verifying the invariance of classification 
properties characterized numerically.  We exemplify with a quick proof of the following statement
(\cite{hn} Proposition 3.1):
\emph{If~ $\D(X) \simeq \D(Y)$, and $X$ is an abelian variety, then so is $Y$}. 
Indeed, the derived invariance of the pluricanonical series \cite{orlov} Corollary 2.1.9 and Theorem \ref{isogeny} imply that $P_1(Y) = P_2 (Y) = 1$ and $q(Y)  = \dim Y$. The main result of \cite{ch} implies that $Y$ is birational, so it actually has a birational morphism, to an abelian variety $B$. But $\omega_X \simeq \OO_X$, so $\omega_Y \simeq \OO_Y$ as well (see e.g. \cite{huybrechts} Proposition 4.1), and therefore $Y \simeq B$. 
\end{remark}

\providecommand{\bysame}{\leavevmode\hbox
to3em{\hrulefill}\thinspace}


\begin{thebibliography}{EMS}

\bibitem[BK]{bk}
S. Barannikov and M. Kontsevich, {Frobenius manifolds and formality of Lie algebras of polyvector fields},
 Internat. Math. Res. Notices no. 4 (1998), 201--215.

\bibitem[Ba]{batyrev}
V. Batyrev, {Stringy Hodge numbers of varieties with Gorenstein canonical singularities},  Integrable systems and algebraic geometry (Kobe/Kyoto, 1997),  World Sci. Publ., River Edge, NJ (1998), 1--32.

\bibitem[BM]{bm}
T. Bridgeland and A. Maciocia, {Complex surfaces with equivalent derived categories}, Math. Z. \textbf{236} (2001), 677--697.

\bibitem[Br1]{brion1}
M. Brion, {Some basic results on actions of non-affine algebraic groups}, preprint 
arXiv:math/0702518.

\bibitem[Br2]{brion2}
M. Brion, {On the geometry of algebraic groups and homogeneous spaces}, preprint
arXiv:math/09095014.

\bibitem[Ca]{caldararu}
A. C\u ald\u araru, {The Mukai pairing, I: The Hochschild structure}, preprint arXiv:math/0308079.

\bibitem[CH]{ch}
J. A. Chen and Ch. Hacon, {Characterization of abelian varieties}, Invent. Math. \textbf{143} (2001), 
435--447.

\bibitem[DL]{dl}
J. Denef and F. Loeser,  {Germs of arcs on singular algebraic varieties and motivic integration}, 
 Invent. Math.  \textbf{135}  (1999),  no. 1, 201--232.

\bibitem[Hu]{huybrechts}
D. Huybrechts, \emph{Fourier-Mukai transforms in algebraic geometry}, Oxford 2006.

\bibitem[HN]{hn}
D. Huybrechts and M. Nieper-Wisskirchen, {Remarks on derived equivalences of Ricci-flat manifolds}, preprint arXiv:0801.4747, to appear in Math. Z. 

\bibitem[Ka]{kawamata}
Y. Kawamata, {$D$-equivalence and $K$-equivalence}, J. Diff. Geom. \textbf{61} (2002), 147--171.

\bibitem[Ko]{kontsevich}
M. Kontsevich, {Homological algebra of mirror symmetry}, Proc. Int. Congr. Math. Z\"urich 1994. Birkh\"auser (1995), 120--139.

\bibitem[Ma]{matsumura}
H.  Matsumura, {On algebraic groups of birational transformations},
Atti della Accademia Nazionale dei Lincei. Rendiconti. Classe di Scienze Fisiche, Matematiche e Naturali. Serie VIII, \textbf{34} (1963), 151--155.

\bibitem[Mu]{mukai}
S. Mukai, {Semi-homogeneous vector bundles on an abelian variety},
J. Math. Kyoto Univ. \textbf{18} (1978), 239--272.

\bibitem[Or]{orlov}
D. Orlov, {Derived categories of coherent sheaves and the equivalences between them}, Russian Math. Surveys \textbf{58} (2003), 511--591.

\bibitem[Ph]{pham}
T. Pham, in preparation.

\bibitem[Ro]{rouquier}
R. Rouquier, {Automorphismes, graduations et cat\'egories triangul\'ees}, available at http://people.maths.ox.ac.uk/$\tilde{}$ rouquier/papers/autograd.pdf, preprint 2009.

\bibitem[Se1]{serre1}
J.-P. Serre, {Espaces fibr\'es alg\'ebriques}, S\'eminaire Claude Chevalley (1958), Expos\'e 
No. 1, Documents Math\'ematiques \textbf{1}, Soc. Math. France, Paris, 2001. 

\bibitem[Se2]{serre2}
J.-P. Serre, {Morphismes universels et vari\'et\'e d`Albanese}, S\'eminaire Claude Chevalley 
(1958--1959), Expos\'e No. 10, Documents Math\'ematiques \textbf{1}, Soc. Math. France, Paris, 2001. 



\end{thebibliography}
\end{document}